\documentclass{amsart}

\usepackage{ amsthm, amssymb, amsmath, latexsym }
\usepackage{ array }
\usepackage{ booktabs }
\usepackage{ cancel }
\usepackage{ url }

\theoremstyle{definition}

\newtheorem{definition}{Definition}[section]
\newtheorem{notation}[definition]{Notation}
\newtheorem{example}[definition]{Example}
\newtheorem{remark}[definition]{Remark}

\theoremstyle{plain}

\newtheorem{lemma}[definition]{Lemma}

\newtheorem{theorem}[definition]{Theorem}

\newcommand{\calO}{\mathcal{O}}
\newcommand{\calQ}{\mathcal{Q}}

\newcommand{\opl}{\dashv}
\newcommand{\opr}{\vdash}

\allowdisplaybreaks

\begin{document}

\title[Self-dual nonsymmetric operads]
{Self-dual nonsymmetric operads\\with two binary operations}

\author{Murray Bremner}

\address{Department of Mathematics and Statistics, University of Saskatchewan, Saskatoon, Canada}

\email{bremner@math.usask.ca}

\author{Juana S\'anchez-Ortega}

\address{Department of Mathematics and Applied Mathematics, University of Cape Town, South Africa}

\email{juana.sanchez-ortega@uct.ac.za}





\begin{abstract}
We consider nonsymmetric operads with two binary operations satisfying relations in arity 3; 
hence these operads are quadratic, and so we can investigate Koszul duality.
We first consider operations which are nonassociative (not necessarily associative) 
and then specialize to the associative case.
We obtain a complete classification of self-dual quadratic nonsymmetric operads with two 
(associative or nonassociative) binary operations.
These operads generalize associativity for one operation to the setting of two operations.
\end{abstract}

\maketitle


\section{Introduction}

\subsection{Associativity for one operation}

Let $\mathcal{O}$ be the free nonsymmetric operad generated by one binary operation $\bullet$ over
the field $\mathbb{F}$:
\[
\mathcal{O} = \bigoplus_{w \ge 0} \mathcal{O}(w),
\qquad
\dim_\mathbb{F} \mathcal{O}(w) = \frac{1}{w{+}1} \binom{2w}{w}.
\]
A basis of $\mathcal{O}(w)$ consists of all complete rooted binary trees with $w$ internal nodes,
and hence $w{+}1$ leaves.
(Note that we are indexing by weight, not arity.)
We interpret these trees as \emph{association types} (placements of parentheses) for the composition 
of $w{+}1$ arguments: the internal nodes represent the operation $\bullet$ and the leaves represent 
the arguments $x_1, \dots, x_{w+1}$.
Since $\mathcal{O}$ is a nonsymmetric operad, we do not need to specify the arguments:
the $i$-th argument is always $x_i$ for $1 \le i \le w{+}1$.
We call this set of trees the \emph{monomial basis} of $\mathcal{O}(w)$;
its size is the Catalan number.
Any (nonzero) element of $\mathcal{O}(w)$ is called a \emph{relation of weight} $w$.
To illustrate, we list the monomial bases for $0 \le w \le 3$, using dash to represent an argument:
\begin{align*}
w = 0\colon 
\quad 
&
-
\qquad
w = 1\colon 
\quad 
- \,\bullet\, -
\qquad
w = 2\colon 
\quad 
( - \,\bullet\, - ) \,\bullet\, -, \quad 
- \,\bullet\, ( - \,\bullet\, - )
\\
w = 3\colon 
\quad
&
( ( - \,\bullet\, - ) \,\bullet\, - ) \,\bullet\, -, \quad 
( - \,\bullet\, ( - \,\bullet\, - ) ) \,\bullet\, -, \quad
( - \,\bullet\, - ) \,\bullet\, ( - \,\bullet\, - ),
\\
&
{-} \,\bullet\, ( ( - \,\bullet\, - ) \,\bullet\, - ), \quad 
- \,\bullet\, ( - \,\bullet\, ( - \,\bullet\, - ) ).
\end{align*}
Given basis monomials $m_1$, $m_2$ of weights $w_1$, $w_2$ we define the \emph{composition}
$m_1 \circ_i m_2$ for $1 \le i \le w_1{+}1$ to be the result of substituting $m_2$ for the $i$-th
argument of $m_1$; in terms of trees, we are identifying the root of $m_2$ with the $i$-th leaf
of $m_1$.

A quadratic relation ($w = 2$) has the form $a m_1 + b m_2 \equiv 0$ with $a, b \in \mathbb{F}$ 
(not both 0) where $m_1 = ( - \,\bullet\, - ) \,\bullet\, -$,
$m_2 = - \,\bullet\, ( - \,\bullet\, - )$.
We define a symmetric bilinear form $\langle \, , \, \rangle$ on $\mathcal{O}(2)$ by
$\langle m_1, m_1 \rangle =  1$,
$\langle m_2, m_2 \rangle = -1$,
$\langle m_1, m_2 \rangle =  0$;
see Loday \cite[Proposition B.3]{Loday2001}.
With respect to this form, the orthogonal complement of the subspace spanned by $a m_1 + b m_2$ 
has basis $b m_1 + a m_2$.
The operad is (Koszul) self-dual (see \S\ref{koszulsubsection}) if and only if 
these two subspaces coincide; equivalently, the matrix 
$R = \begin{bmatrix} a & b \\ b & a \end{bmatrix}$
has rank 1.
This holds if and only if $\det R = a^2 - b^2 = 0$ (and $a, b$ are not both 0).
Up to scalar multiples, the only solutions are $(a,b) = (1,1)$ and $(a,b) = (1,-1)$, which define
the anti-associative operad $m_1 + m_2 \equiv 0$ and the associative operad $m_1 - m_2 \equiv 0$.

By a result of Osborn \cite[Corollary 2]{Osborn} we know that a homogeneous polynomial identity
is satisfied by a unital algebra if and only if the sum of its coefficients is 0.
This condition distinguishes associativity from anti-associativity: only the former
is satisfied by unital algebras.
We have therefore classified self-dual nonsymmetric operads with one binary operation,
and those which define unital algebras.

Our goal in this paper is to extend this classification to operads with two binary operations. 
This allows us to determine generalizations of associativity for these operads,
in the sense that the relations define a self-dual nonsymmetric operad, 
and in every relation the sum of coefficients is 0.
We use an approach based on computer algebra; our main tools are linear algebra over polynomial rings 
and Gr\"obner bases for polynomial ideals \cite[Chs.~7-10]{BD}.
Throughout, we assume that all vector spaces are over the field $\mathbb{F}$ which is algebraically closed 
of characteristic 0.

\subsection{Operads with two binary operations}

To motivate our study of structures with two binary operations, we recall the most important examples.
In his encyclopedia of algebras, Zinbiel (Loday) \cite{Zinbiel} mentions a number of
algebraic operads with two binary operations satisfying relations which are quadratic 
(each monomial has two operations and three arguments) and nonsymmetric 
(each monomial has the identity permutation of the arguments).
Some are (Koszul) self-dual, but most are not.
In most cases, both operations are associative; we denote them by $\opr$ and $\opl$.

\begin{definition} \label{2opdefinition}
\textbf{Two-associative} algebras satisfy only associativity:
  \[
  ( a \opr b ) \opr c \equiv a \opr ( b \opr c ), \qquad
  ( a \opl b ) \opl c \equiv a \opl ( b \opl c ).  
  \]
\textbf{Dual two-associative} algebras satisfy associativity and these relations:
  \[
  ( a \opr b ) \opl c \equiv 0, \qquad
  ( a \opl b ) \opr c \equiv 0, \qquad
  a \opr ( b \opl c ) \equiv 0, \qquad
  a \opl ( b \opr c ) \equiv 0.
  \]
\textbf{Duplicial} algebras satisfy associativity and inner associativity: 
  \[
  ( a \opr b ) \opl c \equiv a \opr ( b \opl c ).
  \]
\textbf{Dual duplicial} algebras satisfy associativity, inner associativity, and:
  \[
  ( a \opl b ) \opr c \equiv 0, \qquad
  a \opl ( b \opr c ) \equiv 0.
  \]
\textbf{Completely associative} algebras satisfy the following relations which include 
associativity and define a self-dual operad:
  \[
  ( a \ast b ) \ast' c \equiv a \ast ( b \ast' c ), \qquad \ast, \ast' \in \{ \opr, \opl \}.
  \]
\textbf{Two-compatible} algebras satisfy associativity and the relation which states that any linear 
combination of the operations is associative:
  \[
  ( a \opr b ) \opl c + ( a \opl b ) \opr c \equiv
  a \opr ( b \opl c ) + a \opl ( b \opr c ).
  \]
\textbf{Dual two-compatible} algebras satisfy the relations of completely associative and two-compatible algebras.
\textbf{Diassociative} algebras (or associative dialgebras) satisfy associativity, inner associativity, and the 
left and right bar relations:
  \[
  ( a \opr b ) \opr c \equiv ( a \opl b ) \opr c, \qquad
  a \opl ( b \opr c ) \equiv a \opl ( b \opl c ).
  \]
The dual operad defines \textbf{dendriform} algebras which have nonassociative operations satisfying
inner associativity and these relations:
\[
( a \opl b ) \opl c \equiv a \opl ( b \opl c ) + a \opl ( b \opr c ),
\qquad
a \opr ( b \opr c ) \equiv ( a \opr b ) \opr c + ( a \opl b ) \opr c.
\]
\end{definition}


\section{Binary operations and Koszul duality}

\subsection{Binary operations}

We write $\calO$ for the free nonsymmetric operad generated by two (nonassociative) binary operations 
$\opr$ and $\opl$ which form a basis of the space $\calO(1)$ of all binary operations.
For $w \ge 0$, a basis of $\mathcal{O}(w)$ consists of all complete rooted binary trees with 
$w$ internal nodes each labelled by an operation.

\begin{lemma}
We have 
  \[
  \dim \calO(w) =  \frac{2^w}{w{+}1} \binom{2w}{w}
  \]
\end{lemma}

\begin{proof}
The factor $2^w$ represents the choices of operation symbols.
\end{proof}

\begin{example}
We have 
$\dim\,\calO(0) = 1$ with basis $\{ \, - \, \}$ (the argument symbol), and
$\dim\,\calO(1) = 2$ with ordered basis $\{ \, - \opr -, \, - \opl - \, \}$.
Every quadratic relation is an element of $\calO(2)$ which has dimension 8 and the ordered basis
in Table \ref{O2basis}.
\end{example}

\vspace{-4mm}

\begin{table}[ht]
\[
\begin{array}{l}
( - \opr - ) \opr -, \qquad
( - \opr - ) \opl -, \qquad
( - \opl - ) \opr -, \qquad
( - \opl - ) \opl -,
\\
- \opr ( - \opr - ), \qquad
- \opr ( - \opl - ), \qquad
- \opl ( - \opr - ), \qquad
- \opl ( - \opl - ).
\end{array}
\]
\caption{Ordered basis of quadratic space $\calO(2)$ for two binary operations}
\label{O2basis}
\end{table}

\vspace{-6mm}

\begin{definition}
A (nonzero) element $\rho \in \calO(2)$ is called a \textbf{quadratic relation}, and a subspace 
$R \subseteq \calO(2)$ is a \textbf{space of quadratic relations}.
The \textbf{operad ideal} $(R)$ generated by a subspace $R \subseteq \calO(w)$ is the smallest 
subspace of $\calO$ which contains $R$ and is closed under composition by arbitrary 
elements of $\calO$.
\end{definition}

The elements of a space $R$ of quadratic relations are satisfied by the quotient operad 
$\calQ = \calO/( R )$ where $( R )$ is the operad ideal generated by $R$.
If $\dim R = r$ then $R$ is the row space of a unique $r \times \dim\calO(2)$ matrix denoted $[R]$ which
has full rank and is in row canonical form (RCF); the columns are labelled by the ordered 
basis in Table \ref{O2basis}.
Conversely, the row space of any matrix with 8 columns can be regarded as a space of
quadratic relations.

\begin{definition}
The matrix $[R]$ is the \textbf{relation matrix} of the quadratic operad $\calQ = \calO/( R )$,  
and its rank $r$ is the \textbf{relation rank} of $\calQ$.
\end{definition}

\subsection{Koszul duality} \label{koszulsubsection}

Loday \cite[Proposition B.3]{Loday2001}, see also \cite[Chapter 7]{LV},
has shown that Koszul duality for
binary operations can be defined in elementary terms, using a nondegenerate inner product 
$\langle -, - \rangle$ on $\calO(2)$.
For $n$-ary operations, see \cite[\S2]{MR}.

\begin{definition}  \label{innerproduct}
For all $\bullet_1, \bullet_2 \in \{ \opr, \opl \}$ we define the symmetric bilinear form
$\langle -, - \rangle$ on basis monomials in Table \ref{O2basis} as follows: 
  \[
  \begin{array}{r@{\;}r@{\;}r@{}}
  \langle
  \; 
  ( - \, \bullet_1 \, - ) \, \bullet_2 \, -, 
  \; 
  - \, \bullet_1 \, ( - \, \bullet_2 \, - ) 
  \; 
  \rangle & = & 0,
  \\
  \langle
  \; 
  ( - \, \bullet_1 \, - ) \, \bullet_2 \, -, 
  \; 
  ( - \, \bullet_1 \, - ) \, \bullet_2 \, - 
  \; 
  \rangle & = & 1,
  \\
  \langle
  \; 
  - \, \bullet_1 \, ( - \, \bullet_2 \, - ), 
  \; 
  - \, \bullet_1 \, ( - \, \bullet_2 \, - ) 
  \; 
  \rangle & = & -1.
  \end{array}
  \]
\end{definition}

\begin{notation}
For any subspace $R \subseteq \calO(2)$ we write $R^\pm$ for its orthogonal complement with respect 
to the symmetric bilinear form $\langle -, - \rangle$ of Definition \ref{innerproduct}.
We write $R^\perp$ for its orthogonal complement with respect to the Euclidean inner product for 
which the monomials in Table \ref{O2basis} are an orthonormal basis.
\end{notation}

\begin{definition}
If $\calQ = \calO / ( R )$ then its \textbf{Koszul dual} is $\calQ^! = \calO / ( R^\pm )$.
We say that $\calQ$ is \textbf{self-dual} if $\calQ = \calQ^!$ (equivalently $R^\pm = R$).
\end{definition}

\begin{lemma}
We have $\dim\,R + \dim\,R^\pm = 8$.
If $\calQ = \calQ^!$ then $\dim\,R = 4$.
\end{lemma}

\begin{proof}
$\langle -, - \rangle$ is nondegenerate and if $\calQ = \calQ^!$ then $\dim\,R = \dim\,R^\pm$.
\end{proof}

\begin{remark}
The relation matrices $[R]$ have entries in the polynomial ring $\Phi = \mathbb{F}[x_1,\dots,x_p]$,
so we regard operads as modules over $\Phi$, not vector spaces over $\mathbb{F}$.
\end{remark}

Loday has shown that computing $R^\pm$ can be reduced to computing the Euclidean orthogonal complement 
of a modified space; see Table \ref{lodayalg}.

\begin{table}[ht]
\begin{center}
\fbox{%
\parbox{12 cm}{%
\emph{Input}: The relation matrix $[R]$ for the quadratic operad $\calQ = \calO/(R)$ where 
$R \subseteq \calO(2)$ and the entries of $[R]$ belong to $\Phi$.

\smallskip

\emph{Output}: The relation matrix $[R^\pm]$ for the Koszul dual $\calQ^! = \calO/(R^\pm)$.

\smallskip

\emph{Algorithm}: 

\smallskip

(1)
Since the last 4 monomials in Table \ref{O2basis} have association type 2, 
we multiply columns 5--8 of $[R]$ by $-1$ to obtain the matrix $[R']$.

\smallskip

(2)
Since any leading 1 of $[R]$ in position $(i,j)$ for $j \ge 5$ becomes $-1$ in $[R']$,
we multiply any such rows of $[R']$ by $-1$ to obtain $[R'']$.

\smallskip

(3)
To find a basis for $R^\pm = (R'')^\perp$, we solve the linear system $[R''] X = 0$.
If $\dim R = r$ then there are $8{-}r$ free variables; we set them equal to the $8{-}r$ standard basis 
vectors in $\mathbb{F}^{8-r}$ and solve for the leading variables.  

\smallskip

(4)
Construct the $(8{-}r) \times 8$ matrix $[R^\pm]$ whose rows are the basis vectors for $(R'')^\perp$ computed 
in step (3).

\medskip \hrule \medskip

To find conditions for self-duality when $r = 4$ we do two more steps:  

\smallskip

(5)
Stack $[R]$ onto $[R^\pm]$, and use the leading 1s of the upper block $[R]$ to reduce the rows of 
the lower block $[R^\pm]$, obtaining the matrix $[T]$:
  \begin{center}
  $
  \left[ \begin{array}{c} R \\ R^\pm \end{array} \right]
  \xrightarrow{\; \text{$[R]$ reduces $[R^\pm]$ to $[T]$} \;}
  \left[ \begin{array}{c} R \\ T \end{array} \right]
  $
  \end{center}

(6)
The operad is self-dual if and only if $R = R^\pm$; that is, $[T] = 0$.
Since the entries of $[T]$ are elements of $\Phi$, we find
the zero set of the ideal generated by $[T]$; 
these values of the parameters $x_1, \dots, x_p$ define self-dual operads.
}
}
\end{center}
\smallskip
\caption{Loday's algorithm for the Koszul dual relation matrix}
\label{lodayalg}
\end{table}


\section{Self-duality for two nonassociative operations}

By \emph{nonassociative} in this section we mean \emph{not necessarily associative}: we do not explicitly
assume associativity, but we will in \S\ref{associativesection}.

\subsection{Computational methods}

In a matrix in RCF, the entries above, below, and to the left 
of each leading 1 are 0, and the remaining entries are free parameters.
For an $r \times n$ matrix, there are $\binom{n}{r}$ choices of columns $j_1 < \cdots < j_r$ 
for leading 1s.
In particular, for $r = 4$ and $n = 8$ we have 70 cases. 

\begin{definition}
We define the \textbf{lex order} on subsets by 
$\{ j_1, \dots, j_r \} \prec \{ j'_1, \dots, j'_r \}$ if and only if $j_k < j'_k$ 
where $k$ is the least index for which $j_k \ne j'_k$.
\end{definition}

\begin{lemma}
The number of parameters as a function of $\{ j_1, \dots, j_r \}$ is:
  \[
  p = p( j_1, \dots, j_r ) = \sum_{i=1}^r \big[(n-j_i)-(r-i)\big].
  \]
\end{lemma}

\begin{proof}
Add the number of entries to the right of each leading 1, and subtract the number of entries which belong 
to the column of another leading 1.
\end{proof}

\begin{lemma} \label{lemma35notself-dual}
For the 35 subsets $\{ j_1, \dots, j_4 \}$ not containing 1 
the matrices $[R]$ do not define self-dual operads for any values of the parameters.
\end{lemma}

\begin{proof}
Starting with $[R]$, we compute $[R']$, $[R'']$, $[R^\pm]$, $[T]$.
Since column 1 of $[R]$ is 0, column 1 of $[R^\pm]$ contains a leading 1.
This leading 1 remains unchanged when we compute $[T]$.
But if $[T]$ contains 1 then its entries generate the unit ideal $\Phi$.
\end{proof}

\begin{example}
For $\{ j_1, \dots, j_4 \} = \{2,3,4,5\}$ we obtain:
  \begin{alignat*}{2}
  [R] &=
  \left[ 
  \begin{array}{@{\,}r@{\;\;}r@{\;\;}r@{\;\;}r@{\;\;}r@{\;\;}r@{\;\;}r@{\;\;}r@{\,}}
  0 & 1 & 0 & 0 & 0 & A & B & C \\
  0 & 0 & 1 & 0 & 0 & D & E & F \\
  0 & 0 & 0 & 1 & 0 & G & H & I \\
  0 & 0 & 0 & 0 & 1 & J & K & L
  \end{array}
  \right]
  &\qquad
  [R'] &=
  \left[ 
  \begin{array}{@{\,}r@{\;\;}r@{\;\;}r@{\;\;}r@{\;\;}r@{\;\;}r@{\;\;}r@{\;\;}r@{\,}}
  0 & 1 & 0 & 0 & 0 & -A & -B & -C \\
  0 & 0 & 1 & 0 & 0 & -D & -E & -F \\
  0 & 0 & 0 & 1 & 0 & -G & -H & -I \\
  0 & 0 & 0 & 0 & -1 & -J & -K & -L
  \end{array}
  \right]
  \\
  [R''] &=
  \left[ 
  \begin{array}{@{\,}r@{\;\;}r@{\;\;}r@{\;\;}r@{\;\;}r@{\;\;}r@{\;\;}r@{\;\;}r@{\,}}
  0 & 1 & 0 & 0 & 0 & -A & -B & -C \\
  0 & 0 & 1 & 0 & 0 & -D & -E & -F \\
  0 & 0 & 0 & 1 & 0 & -G & -H & -I \\
  0 & 0 & 0 & 0 & 1 & J & K & L
  \end{array}
  \right]
  &\qquad
  [R^\pm] &=
  \left[ 
  \begin{array}{@{\,}r@{\;\;}r@{\;\;}r@{\;\;}r@{\;\;}r@{\;\;}r@{\;\;}r@{\;\;}r@{\,}}
  1 & 0 & 0 & 0 & 0 & 0 & 0 & 0 \\
  0 & A & D & G & -J & 1 & 0 & 0 \\
  0 & B & E & H & -K & 0 & 1 & 0 \\
  0 & C & F & I & -L & 0 & 0 & 1
  \end{array}
  \right]
  \end{alignat*}
The resulting matrix $[T]$ is
\[
\left[
\begin{array}{@{\,}c@{\,}c@{\,}c@{\,}c@{\,}c@{\;}ccc@{\,}}  
1 & 0 & 0 & 0 & 0 & 0 & 0 & 0 \\
0 & 0 & 0 & 0 & 0 & J^2{-}G^2{-}D^2{-}A^2{+}1 & JK{-}GH{-}DE{-}AB & JL{-}GI{-}DF{-}AC \\
0 & 0 & 0 & 0 & 0 & JK{-}GH{-}DE{-}AB & K^2{-}H^2{-}E^2{-}B^2{+}1 & KL{-}HI{-}EF{-}BC \\
0 & 0 & 0 & 0 & 0 & JL{-}GI{-}DF{-}AC & KL{-}HI{-}EF{-}BC & L^2{-}I^2{-}F^2{-}C^2{+}1
\end{array}
\right]
\]
Since $[T]$ contains 1 as an entry, no values of the parameters give $[T] = 0$. 
Note that the lower right $3 \times 3$ block is $I-M$ where $M = ( v_i \cdot v_j )$
for $(a_1,\dots,a_4) \cdot (b_1,\dots,b_4) = a_1 b_1 + a_2 b_2 + a_3 b_3 - a_4 b_4$
and $v_1, v_2, v_3$ are the last 3 columns of $[R]$.
\end{example}

\subsection{Cases 1 to 35}

We now consider the subsets $\{ j_1, \dots, j_4 \}$ for which $j_1 = 1$. 

\begin{lemma} \label{lemmacase5}
Of the 35 subsets $\{ 1, j_2, j_3, j_4 \}$ there are 21 for which the ideal generated by 
the entries of $[T]$ equals $\Phi$: in lex order, cases 5, 9, 12--15, 19, 22--35.
For these cases, $[R]$ does not define a self-dual operad for any values of the parameters.
\end{lemma}

\begin{proof}
We give details for case 5; the computations in the other cases are similar.
With leading 1s in columns 1, 2, 3, 8 we obtain:
  \[
  \left[ \begin{array}{c} R \\ R^\pm \end{array} \right]
  =
  \left[ 
  \begin{array}{rrrrrrrr}
  1 & 0 & 0 & A & B & C & D & 0 \\
  0 & 1 & 0 & E & F & G & H & 0 \\
  0 & 0 & 1 & I & J & K & L & 0 \\
  0 & 0 & 0 & 0 & 0 & 0 & 0 & 1 \\ \midrule
  -A & -E & -I & 1 & 0 & 0 & 0 & 0 \\
   B &  F &  J & 0 & 1 & 0 & 0 & 0 \\
   C &  G &  K & 0 & 0 & 1 & 0 & 0 \\
   D &  H &  L & 0 & 0 & 0 & 1 & 0
  \end{array}
  \right]
  \]
Reducing $[R^\pm]$ using $[R]$ produces the matrix $[T]$ whose nonzero columns are:
  \[
  \left[
  \begin{array}{@{\,}cccc@{\,}}
  I^2{+}E^2{+}A^2{+}1 & IJ{+}EF{+}AB & IK{+}EG{+}AC & IL{+}EH{+}AD \\
  {-}IJ{-}EF{-}AB & {-}J^2{-}F^2{-}B^2{+}1 & {-}JK{-}FG{-}BC & {-}JL{-}FH{-}BD \\
  {-}IK{-}EG{-}AC & {-}JK{-}FG{-}BC & {-}K^2{-}G^2{-}C^2{+}1 & {-}KL{-}GH{-}CD \\
  {-}IL{-}EH{-}AD & {-}JL{-}FH{-}BD & {-}KL{-}GH{-}CD & {-}L^2{-}H^2{-}D^2{+}1
  \end{array}
  \right]
  \]
$[T] = 0$ if and only if there exist 4 vectors in $\mathbb{F}^3$, namely the columns
of the $3 \times 4$ block of $[R]$ containing the parameters,
satisfying these equations with respect to the Euclidean inner product:
$v_1 \cdot v_1 = -1$,
$v_2 \cdot v_2 = v_3 \cdot v_3 = v_4 \cdot v_4 = 1$,
$v_i \cdot v_j = 0$ ($i \ne j$).
This says that there exist 4 orthogonal nonzero vectors in $\mathbb{F}^3$; contradiction.
\end{proof}

We can also prove Lemma \ref{lemmacase5} using computer algebra; see Table \ref{grobneralg}.

\begin{table}[ht]
\begin{center}
\fbox{%
\parbox{12 cm}{%
\emph{Input:}
A monomial order $\prec$ on $\Phi = \mathbb{F}[ x_1, \dots, x_p ]$, together with a subset 
$G = \{ f_1 \prec \cdots \prec f_n \} \subset \Phi$ generating the ideal $I \subseteq \Phi$.

\emph{Output:}
The Gr\"obner basis of $I$ with respect to $\prec$.

\emph{Algorithm:}

(1) Set $G_0 \leftarrow \emptyset$ and $G_1 \leftarrow G$.

(2) Set $k \leftarrow 1$.

(3) While $G_{k-1} \ne G_k$ do:

\begin{enumerate}
\item[(a)]
Self-reduce $G_k$:
For each $f_i \in G_k$ do:
\begin{itemize}
\item 
Compute the normal form $N(f_i)$ with respect to the previous elements $f_1 \prec \cdots \prec f_{i-1}$.
\item
If $N(f_i) = 0$ then remove $f_i$ from $G_k$, otherwise replace $f_i$ by $\textsf{monicform}(N(f_i))$.
\item
Sort $G_k$ with respect to $\prec$.
\end{itemize}
\item[(b)]
Compute the set of S-polynomials:
\begin{itemize}
\item
Set $H \leftarrow \emptyset$.
\item
For all $1 \le i < j \le |G_k|$ compute $h_{ij} = S(f_i,f_j)$ and its normal form $N(h_{ij})$
with respect to $G_k$; 
if $N(h_{ij}) \ne 0$ then set $H \leftarrow H \cup \{ \textsf{monicform}(N(h_{ij})) \}$. 
\end{itemize}
\item[(c)]
Set $G_{k+1} \leftarrow G_k \cup H$.
Sort $G_{k+1}$ with respect to $\prec$.
\item[(d)]
Set $k \leftarrow k+1$.
\end{enumerate}
}
}
\end{center}
\caption{Algorithm to compute a Gr\"obner basis for a polynomial ideal}
\label{grobneralg}
\end{table}

\begin{remark}
We apply the algorithm of Table \ref{grobneralg} to the ideal generated by the entries of
the matrix $[T]$ from the proof of Lemma \ref{lemmacase5}:
\begin{enumerate}
\item[$k=1$:]
The original set $G_1$ of 10 generators is already self-reduced; it produces 24 S-polynomials
with $N(h) \ne 0$.
\item[$k=2$:]
The set $G_2$ has 34 elements but self-reduction eliminates 4.
The remaining 30 generators produce 232 S-polynomials with $N(h) \ne 0$.
\item[$k=3$:]
The set $G_3$ has 262 elements but self-reduction eliminates 114.
The remaining 148 generators produce 6916 S-polynomials $h$ with $N(h) \ne 0$.
\item[$k=4$:]
The set $G_4$ has 7064 elements but self-reduction eliminates 6620.
The remaining 444 generators produce 92 S-polynomials $h$ with $N(h) \ne 0$.
\item[$k=5$:]
The set $G_5$ has 536 elements but self-reduction eliminates 523.
The remaining 13 generators are the 12 parameters $A, \dots, L$ together with 1.
\item[$k=6$:]
The set $G_6$ has 13 elements but self-reduction eliminates 12 and leaves $\{ 1 \}$.
The algorithm terminates with the Gr\"obner basis $\{ 1 \}$.
\end{enumerate}
We did these calculations in Maple with the graded reverse lex order ($A \prec \cdots \prec L$).
\end{remark}

\begin{definition} \label{SPdefinition}
Let $\mathcal{S}$ be the set of 14 cases corresponding to subsets $\{ 1, j_2, j_3, j_4 \}$ 
for which the entries of $[T]$ generate a proper ideal in $\Phi$:
\[
\mathcal{S} = 
\{ \, 1, \; 2, \; 3, \; 4, \; 6, \; 7, \; 8, \; 10, \; 11, \; 16, \; 17, \; 18, \; 20, \; 21 \, \}.
\]
By Lemma \ref{lemmacase5}, these cases have self-dual operads defined by
the parameter values in the zero set of the ideal generated by $[T]$.
The corresponding relation matrices $[R]$ are displayed in Table \ref{14nontrivial}.
We write $P$ for the $4 \times 4$ \textbf{parameter matrix} 
obtained by deleting columns $1, j_2, j_3, j_4$ from $[R]$.
In case 1, $P = [ W, X, Y, Z ]$ where $W = (W_1,W_2,W_3,W_4)^t$, etc.
In the other cases, we obtain $P$ by setting some entries to 0 in $[ W, X, Y, Z ]$;
these entries are a subset of $\{ W_2, W_3, W_4, X_3, X_4, Y_4 \}$.
\end{definition}  

  \begin{table}[ht]
  \[
  \begin{array}{r@{\;\;}r@{\;\;}r}
  1\!\!
  \left[ \begin{array}{@{\,}c@{\;}c@{\;}c@{\;}c@{\;}c@{\;}c@{\;}c@{\;}c@{\,}} 
  1 & 0 & 0 & 0 & W_1 & X_1 & Y_1 & Z_1 \\ 
  0 & 1 & 0 & 0 & W_2 & X_2 & Y_2 & Z_2 \\ 
  0 & 0 & 1 & 0 & W_3 & X_3 & Y_3 & Z_3 \\ 
  0 & 0 & 0 & 1 & W_4 & X_4 & Y_4 & Z_4
  \end{array} \right] 
  & 
  2\!\!
  \left[ \begin{array}{@{\,}c@{\;}c@{\;}c@{\;}c@{\;}c@{\;}c@{\;}c@{\;}c@{\,}} 
  1 & 0 & 0 & W_1 & 0 & X_1 & Y_1 & Z_1 \\ 
  0 & 1 & 0 & W_2 & 0 & X_2 & Y_2 & Z_2 \\ 
  0 & 0 & 1 & W_3 & 0 & X_3 & Y_3 & Z_3 \\ 
  0 & 0 & 0 & 0 & 1 & X_4 & Y_4 & Z_4
  \end{array} \right] 
  & 
  3\!\!
  \left[ \begin{array}{@{\,}c@{\;}c@{\;}c@{\;}c@{\;}c@{\;}c@{\;}c@{\;}c@{\,}} 
  1 & 0 & 0 & W_1 & X_1 & 0 & Y_1 & Z_1 \\ 
  0 & 1 & 0 & W_2 & X_2 & 0 & Y_2 & Z_2 \\ 
  0 & 0 & 1 & W_3 & X_3 & 0 & Y_3 & Z_3 \\ 
  0 & 0 & 0 & 0 & 0 & 1 & Y_4 & Z_4
  \end{array} \right] 
  \\[8mm]
  4\!\!
  \left[ \begin{array}{@{\,}c@{\;}c@{\;}c@{\;}c@{\;}c@{\;}c@{\;}c@{\;}c@{\,}} 
  1 & 0 & 0 & W_1 & X_1 & Y_1 & 0 & Z_1 \\ 
  0 & 1 & 0 & W_2 & X_2 & Y_2 & 0 & Z_2 \\ 
  0 & 0 & 1 & W_3 & X_3 & Y_3 & 0 & Z_3 \\ 
  0 & 0 & 0 & 0 & 0 & 0 & 1 & Z_4
  \end{array} \right] 
  & 
  6\!\!
  \left[ \begin{array}{@{\,}c@{\;}c@{\;}c@{\;}c@{\;}c@{\;}c@{\;}c@{\;}c@{\,}} 
  1 & 0 & W_1 & 0 & 0 & X_1 & Y_1 & Z_1 \\ 
  0 & 1 & W_2 & 0 & 0 & X_2 & Y_2 & Z_2 \\ 
  0 & 0 & 0 & 1 & 0 & X_3 & Y_3 & Z_3 \\ 
  0 & 0 & 0 & 0 & 1 & X_4 & Y_4 & Z_4
  \end{array} \right] 
  & 
  7\!\!
  \left[ \begin{array}{@{\,}c@{\;}c@{\;}c@{\;}c@{\;}c@{\;}c@{\;}c@{\;}c@{\,}} 
  1 & 0 & W_1 & 0 & X_1 & 0 & Y_1 & Z_1 \\ 
  0 & 1 & W_2 & 0 & X_2 & 0 & Y_2 & Z_2 \\ 
  0 & 0 & 0 & 1 & X_3 & 0 & Y_3 & Z_3 \\ 
  0 & 0 & 0 & 0 & 0 & 1 & Y_4 & Z_4
  \end{array} \right] 
  \\[8mm]
  8\!\!
  \left[ \begin{array}{@{\,}c@{\;}c@{\;}c@{\;}c@{\;}c@{\;}c@{\;}c@{\;}c@{\,}} 
  1 & 0 & W_1 & 0 & X_1 & Y_1 & 0 & Z_1 \\ 
  0 & 1 & W_2 & 0 & X_2 & Y_2 & 0 & Z_2 \\ 
  0 & 0 & 0 & 1 & X_3 & Y_3 & 0 & Z_3 \\ 
  0 & 0 & 0 & 0 & 0 & 0 & 1 & Z_4
  \end{array} \right] 
  & 
  10\!\!
  \left[ \begin{array}{@{\,}c@{\;}c@{\;}c@{\;}c@{\;}c@{\;}c@{\;}c@{\;}c@{\,}} 
  1 & 0 & W_1 & X_1 & 0 & 0 & Y_1 & Z_1 \\ 
  0 & 1 & W_2 & X_2 & 0 & 0 & Y_2 & Z_2 \\ 
  0 & 0 & 0 & 0 & 1 & 0 & Y_3 & Z_3 \\ 
  0 & 0 & 0 & 0 & 0 & 1 & Y_4 & Z_4
  \end{array} \right] 
  &
  11\!\!
  \left[ \begin{array}{@{\,}c@{\;}c@{\;}c@{\;}c@{\;}c@{\;}c@{\;}c@{\;}c@{\,}} 
  1 & 0 & W_1 & X_1 & 0 & Y_1 & 0 & Z_1 \\ 
  0 & 1 & W_2 & X_2 & 0 & Y_2 & 0 & Z_2 \\ 
  0 & 0 & 0 & 0 & 1 & Y_3 & 0 & Z_3 \\ 
  0 & 0 & 0 & 0 & 0 & 0 & 1 & Z_4
  \end{array} \right] 
  \\[8mm]
  16\!\!
  \left[ \begin{array}{@{\,}c@{\;}c@{\;}c@{\;}c@{\;}c@{\;}c@{\;}c@{\;}c@{\,}} 
  1 & W_1 & 0 & 0 & 0 & X_1 & Y_1 & Z_1 \\ 
  0 & 0 & 1 & 0 & 0 & X_2 & Y_2 & Z_2 \\ 
  0 & 0 & 0 & 1 & 0 & X_3 & Y_3 & Z_3 \\ 
  0 & 0 & 0 & 0 & 1 & X_4 & Y_4 & Z_4
  \end{array} \right] 
  & 
  17\!\!
  \left[ \begin{array}{@{\,}c@{\;}c@{\;}c@{\;}c@{\;}c@{\;}c@{\;}c@{\;}c@{\,}} 
  1 & W_1 & 0 & 0 & X_1 & 0 & Y_1 & Z_1 \\ 
  0 & 0 & 1 & 0 & X_2 & 0 & Y_2 & Z_2 \\ 
  0 & 0 & 0 & 1 & X_3 & 0 & Y_3 & Z_3 \\ 
  0 & 0 & 0 & 0 & 0 & 1 & Y_4 & Z_4
  \end{array} \right] 
  & 
  18\!\!
  \left[ \begin{array}{@{\,}c@{\;}c@{\;}c@{\;}c@{\;}c@{\;}c@{\;}c@{\;}c@{\,}} 
  1 & W_1 & 0 & 0 & X_1 & Y_1 & 0 & Z_1 \\ 
  0 & 0 & 1 & 0 & X_2 & Y_2 & 0 & Z_2 \\ 
  0 & 0 & 0 & 1 & X_3 & Y_3 & 0 & Z_3 \\ 
  0 & 0 & 0 & 0 & 0 & 0 & 1 & Z_4
  \end{array} \right] 
  \\[8mm]
  20\!\!
  \left[ \begin{array}{@{\,}c@{\;}c@{\;}c@{\;}c@{\;}c@{\;}c@{\;}c@{\;}c@{\,}} 
  1 & W_1 & 0 & X_1 & 0 & 0 & Y_1 & Z_1 \\ 
  0 & 0 & 1 & X_2 & 0 & 0 & Y_2 & Z_2 \\ 
  0 & 0 & 0 & 0 & 1 & 0 & Y_3 & Z_3 \\ 
  0 & 0 & 0 & 0 & 0 & 1 & Y_4 & Z_4
  \end{array} \right] 
  & 
  21\!\!
  \left[ \begin{array}{@{\,}c@{\;}c@{\;}c@{\;}c@{\;}c@{\;}c@{\;}c@{\;}c@{\,}} 
  1 & W_1 & 0 & X_1 & 0 & Y_1 & 0 & Z_1 \\ 
  0 & 0 & 1 & X_2 & 0 & Y_2 & 0 & Z_2 \\ 
  0 & 0 & 0 & 0 & 1 & Y_3 & 0 & Z_3 \\ 
  0 & 0 & 0 & 0 & 0 & 0 & 1 & Z_4
  \end{array} \right] 
  \end{array}
  \]
  \caption{The 14 relation matrices $[R]$ defining self-dual operads}
  \label{14nontrivial}
  \end{table}
  
\begin{remark}
As the parameters range over $\mathbb{F}$, each matrix in Table \ref{14nontrivial} defines
a Schubert cell in the Grassmannian $G_\mathbb{F}(4,8)$ of 4-dimensional subspaces of $\mathbb{F}^8$.
The nonzero entries of $P$ (rotated $90^\circ$ counter-clockwise)
form a Young diagram for a partition of the number of parameters.
For example, case 1 gives $16 = 4{+}4{+}4{+}4$, and case 21 gives $10 = 4{+}3{+}2{+}1$.
For more information, see Fulton \cite{Fulton}, Hiller \cite{Hiller}.
\end{remark}
    
\begin{lemma} \label{sylvesterlemma}
Assume that $\mathbb{F} = \mathbb{R}$.
Let $D$ and $E$ be $n \times n$ diagonal matrices with nonzero entries $\pm 1$.
Every solution of $A^t E A = D$ has the form
$A = \overline{\sqrt{D}} C \sqrt{E}$
for some orthogonal matrix $C$ where bar denotes complex conjugate.
\end{lemma}  
  
\begin{proof}
If $D \ne I$ or $E \ne I$ then we extend scalars to $\mathbb{C}$ so that 
we can form $\sqrt{D}$ and $\sqrt{E}$.
Let $S$ be any real symmetric matrix with the same signature as $E$.
By Sylvester's Law of Inertia there is an orthogonal matrix $C$ for which $C^t S C = E$.
Thus $S$ has the square root $\sqrt{S} = C \sqrt{E} C^t$ and hence $\sqrt{S}^t = \sqrt{S}$.
It follows that $B = \sqrt{S} C = C \sqrt{E}$ is a solution of $B^t B = E$ 
and every solution can be obtained this way for some orthogonal matrix $C$.
Clearly, $B$ satisfies $B^t B = E$ if and only if $A = (\sqrt{D})^{-1} B$ satisfies $A^t D A = E$.
Finally, note that $(\sqrt{D})^{-1} = \overline{\sqrt{D}}$.
\end{proof}  

\begin{remark} \label{DEremark}
If $A = ( a_{ij} )$ is a solution of $A^t E A = D$ then the columns $U_1, \dots, U_n$ of $A$ 
are basis of $\mathbb{R}^n$ such that 
$\langle U_i, U_i \rangle = d_{ii}$ and
$\langle U_i, U_j \rangle = 0$ ($i \ne j$)
where the diagonal entries of $E$ are the signature:
$\langle ( v_1, \dots, v_n ), ( w_1, \dots, w_n ) \rangle = \sum_{i=1}^n e_{ii} v_i w_i$.
\end{remark}

\begin{theorem} \label{theorem14nonassociative}
Assume that $\mathbb{F} = \mathbb{R}$.
For every case in $\mathcal{S}$, the parameter values defining self-dual operads
are the solutions of the equation $P^t D P = E$, where $P$ is the parameter matrix
and the diagonal matrices $D, E$ are as follows:

$\bullet$
Case 1: $D =  E = I_4$.

$\bullet$
Cases 2, 3, 4, 6, 8, 16, 17, 18: $D = \mathrm{diag}(1,1,1,{-}1)$, $E = \mathrm{diag}({-}1,1,1,1)$.

$\bullet$
Cases 10, 11, 20, 21: $D = \mathrm{diag}(1,1,{-}1,{-}1)$, $E = \mathrm{diag}({-}1,{-}1,1,1)$.

\noindent
Thus $P = \overline{\sqrt{D}} C \sqrt{E}$ where $C$ is orthogonal and has zeros in the same
entries as $P$.
\end{theorem}

\begin{proof}
We verify the claims case-by-case.

\noindent $\bullet$
\emph{Case 1}:
This case can be solved in terms of 4-dimensional Euclidean geometry.
The relation matrix is $[ \, R \, ] = [ \, I_4 \mid P \, ]$ where $P = [ W, X, Y, Z ]$.
Clearly $[ \, R' \, ] = [ \, R'' \, ] = [ \, I_4 \mid -P \, ]$, 
and hence $[ \, R^\pm \, ] = [ \, P^t \mid I_4 \, ]$.
We obtain 
  \[
  \left[ \begin{array}{c} R \\ R^\pm \end{array} \right]
  =  
  \left[ \begin{array}{c|c} I_4 & P \\ \midrule P^t & I_4 \end{array} \right]
  \xrightarrow{\; \text{$[R]$ reduces $[R^\pm]$} \;}
  \left[ \begin{array}{c|c} I_4 & P \\ \midrule O & T' \end{array} \right]
  =
  \left[ \begin{array}{c} R \\ T \end{array} \right]
  \]
In this and the remaining cases we simplify $T$ by making its diagonal entries monic:
we divide row $i$ by the leading coefficient of the $i$-th diagonal entry for $1 \le i \le 4$.
With these sign changes, $T'$ becomes $T''$ where
  \[
  T'' = -T' = P^t P - I =
  \left[ 
  \begin{array}{cccc}
  W \cdot W - 1 & W \cdot X & W \cdot Y & W \cdot Z \\
  W \cdot X & X \cdot X - 1 & X \cdot Y & X \cdot Z \\
  W \cdot Y & X \cdot Y & Y \cdot Y - 1 & Y \cdot Z \\
  W \cdot Z & X \cdot Z & Y \cdot Z & Z \cdot Z - 1
  \end{array}
  \right]
  \]
Hence, in order for the matrix $P$ of parameter values to belong to the zero set of the ideal
generated by $T''$, it is necessary and sufficient that its columns $W, X, Y, Z$
form an orthonormal basis of $\mathbb{R}^4$.

The Gr\"obner basis for the ideal generated by $T''$ has 141 elements,
the greatest of which in lex order is this polynomial of degree 5 with 14 terms:
\begin{align*}
&
W_4X_4Y_3Z_1Z_2-W_3X_4Y_4Z_1Z_2-W_4X_4Y_1Z_2Z_3+W_1X_4Y_4Z_2Z_3-W_4X_2Y_3Z_1Z_4
\\[-1mm]
&
+W_3X_2Y_4Z_1Z_4+W_3X_4Y_1Z_2Z_4-W_1X_4Y_3Z_2Z_4+W_4X_2Y_1Z_3Z_4
\\[-1mm]
&
-W_1X_2Y_4Z_3Z_4-W_3X_2Y_1Z_4^2+W_1X_2Y_3Z_4^2+W_3X_2Y_1-W_1X_2Y_3.
\end{align*}
It is easier to find the zero set from the generators than the Gr\"obner basis.

\noindent $\bullet$
\emph{Case 2}:
We define the symmetric bilinear form $\langle U, V \rangle$ to have signature equal to the diagonal
entries of $D$.
We obtain
\begin{align*}
T'' 
&=
\left[ 
\begin{array}{@{}cccc@{}}
\langle W, W \rangle {+} 1 & \langle W, X \rangle & \langle W, Y \rangle & \langle W, Z \rangle \\
\langle W, X \rangle & \langle X, X \rangle {-} 1 & \langle X, Y \rangle & \langle X, Z \rangle \\
\langle W, Y \rangle & \langle X, Y \rangle & \langle Y, Y \rangle {-} 1 & \langle Y, Z \rangle \\
\langle W, Z \rangle & \langle X, Z \rangle & \langle Y, Z \rangle & \langle Z, Z \rangle {-} 1
\end{array}
\right]
\qquad
W = [W_1,W_2,W_3,0]^t
\\
&= 
P^t D P - E.
\end{align*}
Hence 
$P = \mathrm{diag}(1,1,1,\mp i) \, C \, \mathrm{diag}(\pm i,1,1,1)$ for some orthogonal matrix $C$.
Since $W_4 = 0$ we require that $P$ (and hence $C$) has 0 in the lower left corner.

The Gr\"obner basis for the ideal generated by the entries of $T''$ has 112 elements,
the greatest of which in lex order is this polynomial of degree 7 with 36 terms:
\begin{align*}
&
W_1X_3Y_3Y_4^2Z_2Z_3-W_1X_4Y_4^3Z_2Z_3+W_3X_1Y_2Y_4^2Z_3^2-W_2X_1Y_3Y_4^2Z_3^2
\\[-1mm]
&
-W_1X_2Y_3Y_4^2Z_3^2-W_1X_3Y_3^2Y_4Z_2Z_4+W_1X_3Y_4^3Z_2Z_4-2W_3X_1Y_2Y_3Y_4Z_3Z_4
\\[-1mm]
&
+2W_2X_1Y_3^2Y_4Z_3Z_4+W_1X_2Y_3^2Y_4Z_3Z_4+W_1X_2Y_4^3Z_3Z_4+W_1X_3Y_4Z_2Z_3^2Z_4
\\[-1mm]
&
-W_1X_2Y_4Z_3^3Z_4+W_3X_1Y_2Y_3^2Z_4^2-W_2X_1Y_3^3Z_4^2-W_1X_2Y_3Y_4^2Z_4^2
\\[-1mm]
&
-W_1X_3Y_3Z_2Z_3Z_4^2-2W_1X_4Y_4Z_2Z_3Z_4^2+W_1X_2Y_3Z_3^2Z_4^2+2W_1X_3Y_4Z_2Z_4^3
\\[-1mm]
&
+W_1X_4Y_2Z_3Z_4^3-W_1X_3Y_2Z_4^4+W_3X_1Y_2Y_3^2-W_2X_1Y_3^3
\\[-1mm]
&
+W_2X_1Y_3Y_4^2+W_1X_3Y_3Z_2Z_3-W_1X_4Y_4Z_2Z_3+W_3X_1Y_2Z_3^2
\\[-1mm]
&
-W_1X_3Y_2Z_3^2-W_2X_1Y_3Z_3^2+W_1X_3Y_4Z_2Z_4+W_1X_4Y_2Z_3Z_4
\\[-1mm]
&
+W_1X_2Y_4Z_3Z_4-W_1X_3Y_2Z_4^2+W_2X_1Y_3Z_4^2+W_2X_1Y_3.
\end{align*}

\noindent $\bullet$
\emph{Case 3}:
Define $D, E$ as in Case 2.
Then $T''$ has the same form as Case 2 but
$W = [W_1,W_2,W_3,0]^t$ and
$X = [X_1,X_2,X_3,0]^t$.
The Gr\"obner basis for the ideal generated by $T''$ has 63 elements; the greatest is
\[
W_3X_3Y_1Y_2-W_2X_3Y_1Y_3-W_3X_1Y_2Y_3+W_2X_1Y_3^2+W_2X_1Z_4^2.
\]

\noindent $\bullet$
\emph{Case 4}:
Define $D, E$ as in Case 2.
Then $T''$ has the same form as Case 2 but
$W = [W_1,W_2,W_3,0]^t$,
$X = [X_1,X_2,X_3,0]^t$,
$Y = [Y_1,Y_2,Y_3,0]^t$.
The Gr\"obner basis for the ideal generated by $T''$ has 31 elements;
the greatest is
\[
W_3X_3Y_1Y_2-W_2X_3Y_1Y_3-W_3X_1Y_2Y_3+W_2X_1Y_3^2-W_2X_1.
\]

\noindent $\bullet$
\emph{Case 6}:
Define $D, E$ as in Case 2.
Then $T''$ has the same form as Case 2 but
$W = [W_1,W_2,0,0]^t$.
The Gr\"obner basis for the ideal generated by $T''$ has 72 elements;
the greatest is
\begin{align*}
&
X_4Y_1Y_2Z_3^2-X_1Y_2Y_4Z_3^2+X_4 Z_1 Z_2 Z_3^2-X_1Z_2Z_3^2Z_4-X_4Y_1Y_2Z_4^2+X_1 Y_2 Y_4 Z_4^2
\\
&
-X_4Z_1Z_2Z_4^2+X_1 Z_2 Z_4^3-X_4Y_1Y_2+X_1 Y_2 Y_4+X_1 Z_2 Z_4.
\end{align*}

\noindent $\bullet$
\emph{Case 7}:
Define $D, E$ as in Case 2.
Then $T''$ has the same form as Case 2 but
$W = [W_1,W_2,0,0]^t$ and
$X = [X_1,X_2,X_3,0]^t$.
The Gr\"obner basis for the ideal generated by $T''$ has 50 elements;
the greatest is
$X_1X_2Y_3^2+Y_1Y_2Y_3^2+Y_1Y_2Z_3^2+X_1X_2Z_4^2$.

\noindent $\bullet$
\emph{Case 8}:
Define $D, E$ as in Case 2.
Then $T''$ has the same form as Case 2 but
$W = [W_1,W_2,0,0]^t$,
$X = [X_1,X_2,X_3,0]^t$,
$Y = [Y_1,Y_2,Y_3,0]^t$.
The Gr\"obner basis for the ideal generated by $T''$ has 50 elements;
the greatest is
$X_1X_2Y_3^2+Y_1Y_2Y_3^2-X_1X_2$.

\noindent $\bullet$
\emph{Case 10}:
We now have 
$D = \mathrm{diag}(1,1,{-}1,{-}1)$ and $E = \mathrm{diag}({-}1,{-}1,1,1)$.
We define $\langle U, V \rangle$ to have signature equal to the diagonal
entries of $D$.
We obtain
\begin{align*}
T''
&=
\left[ 
\begin{array}{@{}cccc@{}}
\langle W, W \rangle {+} 1 & \langle W, X \rangle & \langle W, Y \rangle & \langle W, Z \rangle \\
\langle W, X \rangle & \langle X, X \rangle {+} 1 & \langle X, Y \rangle & \langle X, Z \rangle \\
\langle W, Y \rangle & \langle X, Y \rangle & \langle Y, Y \rangle {-} 1 & \langle Y, Z \rangle \\
\langle W, Z \rangle & \langle X, Z \rangle & \langle Y, Z \rangle & \langle Z, Z \rangle {-} 1
\end{array}
\right]
\qquad
\begin{array}{l}
W = [W_1,W_2,0,0]^t, \\
X = [X_1,X_2,0,0]^t
\end{array}
\\
&= 
P^t D P {-} E.
\end{align*}
Lemma \ref{sylvesterlemma} shows that 
$P = \mathrm{diag}(1,1,\mp i,\mp i) \, C \, \mathrm{diag}(\pm i,\pm i,1,1)$ for some orthogonal matrix $C$.
Since $W_3 = W_4 = X_3 = X_4 = 0$ we require that $P$ (and hence $C$) has a $2 \times 2$ zero block
in the lower left corner.
The Gr\"obner basis for the ideal generated by $T''$ has 16 elements;
the greatest is
$W_2X_1X_2-W_1X_2^2-W_1$.

\noindent $\bullet$
\emph{Case 11}:
Define $D, E$ as in Case 10.
Then $T''$ has the same form as Case 10 but
$W = [W_1,W_2,0,0]^t$,
$X = [X_1,X_2,0,0]^t$,
$Y = [Y_1,Y_2,Y_3,0]^t$.
The Gr\"obner basis for the ideal generated by $T''$ has 13 elements;
the greatest is
$W_2X_1X_2-W_1X_2^2-W_1$.

\noindent $\bullet$
\emph{Case 16}:
Define $D, E$ as in Case 2.
Then $T''$ has the same form as Case 2 but
$W = [W_1,0,0,0]^t$.
The Gr\"obner basis for the ideal generated by $T''$ has 31 elements; the greatest is
$X_4Y_4Z_2Z_3-X_3Y_4Z_2Z_4-X_4Y_2Z_3Z_4+X_3Y_2Z_4^2+X_3Y_2$.

\noindent $\bullet$
\emph{Case 17}:
Define $D, E$ as in Case 2.
Then $T''$ has the same form as Case 2 but
$W = [W_1,0,0,0]^t$ and
$X = [X_1,X_2,X_3,0]^t$.
The Gr\"obner basis for the ideal generated by $T''$ has 23 elements; the greatest is
$Y_2Y_3Z_4^2+Z_2Z_3Z_4^2+Y_2Y_3$.

\noindent $\bullet$
\emph{Case 18}:
Define $D, E$ as in Case 2.
Then $T''$ has the same form as Case 2 but
$W = [W_1,0,0,0]^t$,
$X = [X_1,X_2,X_3,0]^t$,
$Y = [Y_1,Y_2,Y_3,0]^t$.
We have $T'' = P^t D P {-} E$ and the rest 
is the same as in Case 2 except that $W_2 = W_3 = W_4 = X_4 = Y_4 = 0$.
The Gr\"obner basis for the ideal generated by the entries of $T''$ has 33 elements,
the greatest of which is
$X_3Y_2Y_3-X_2Y_3^2+X_2$.

\noindent $\bullet$
\emph{Case 20}:
Define $D, E$ as in Case 10.
Then $T''$ has the same form as Case 10 but
$W = [W_1,0,0,0]^t$ and
$X = [X_1,X_2,X_3,0]^t$.
The Gr\"obner basis for the ideal generated by $T''$ has 13 elements; the greatest is
$Y_4Z_3Z_4-Y_3Z_4^2-Y_3$.

\noindent $\bullet$
\emph{Case 21}:
Define $D, E$ as in Case 10.
Then $T''$ has the same form as Case 10 but
$W = [W_1,0,0,0]^t$,
$X = [X_1,X_2,0,0]^t$,
$Y = [Y_1,Y_2,Y_3,0]^t$.
The Gr\"obner basis for the ideal generated by the entries of $T''$ has 10 elements;
the greatest is
$W_1^2+1$.
\end{proof}

\begin{example}
In case 21, where $D, E$ are as in case 10, the Gr\"obner basis is
\[
Z_3, \;\; Z_2, \;\; Z_1, \;\; Y_2, \;\; Y_1, \;\; X_1, \;\; Z_4^2+1, \;\; Y_3^2+1, \;\; X_2^2+1, \;\; W_1^2+1.
\]
From this we immediately obtain the following zero set for the ideal:
  \[
  W = [ \pm i, 0, 0, 0 ], \quad
  X = [ 0, \pm i, 0, 0 ], \quad
  Y = [ 0, 0, \pm i, 0 ], \quad
  Z = [ 0, 0, 0, \pm i ].
  \]
The corresponding relations between the operations $\opr, \opl$ are:
  \begin{alignat*}{2}
  ( a \opr b ) \opr c &= \mp i \, ( a \opr b ) \opl c, 
  &\qquad\qquad
  ( a \opl b ) \opr c &= \mp i \, ( a \opl b ) \opl c,
  \\[-1mm]
  a \opr ( b \opr c ) &= \mp i \, a \opr ( b \opl c ), 
  &\qquad\qquad
  a \opl ( b \opr c ) &= \mp i \, a \opl ( b \opl c ).
  \end{alignat*}
``Whenever the second operation changes, the multiplier $\pm i$ appears.''
\end{example}

\begin{example}
In case 18, for which $D, E$ are as in case 2, the relation matrix is
  \[
  \left[ \begin{array}{@{\;}cccccccc@{\;}} 
  1 & W_1 & 0 & 0 & X_1 & Y_1 & 0 & Z_1 \\ 
  0 & 0 & 1 & 0 & X_2 & Y_2 & 0 & Z_2 \\ 
  0 & 0 & 0 & 1 & X_3 & Y_3 & 0 & Z_3 \\ 
  0 & 0 & 0 & 0 & 0 & 0 & 1 & Z_4
  \end{array} \right] 
  \]
and the Gr\"obner basis is
\begin{align*}
&
Z_3, \;\; Z_2, \;\; Z_1, \;\; Y_1, \;\; X_1, \;\; Z_4^2 + 1, \;\; Y_2^2 + Y_3^2 - 1, \;\; X_2 Y_2 + X_3 Y_3, \;\; X_3^2 + Y_3^2 - 1,
\\
&
X_2 X_3 + Y_2 Y_3, \;\; X_2^2 - Y_3^2, \;\; W_1^2 + 1, \;\; X_3 Y_2 Y_3 - X_2 Y_3^2 + X_2.
\end{align*}
From this we obtain the following one-parameter family of solutions:
\begin{align*}
&
W_1 = \pm i, \quad 
X_1 = 0, \quad 
X_2 = \mathrm{free}, \quad 
X_3 = \pm \sqrt{1-X_2^2}, \quad 
Y_1 = 0, 
\\[-2mm]
&
Y_2 = \mp \sqrt{1-X_2^2}, \quad 
Y_3 = X_2, \quad 
Z_1 = 0, \quad 
Z_2 = 0, \quad 
Z_3 = 0, \quad 
Z_4 = \pm i.
\end{align*}
The second solution is obtained by changing the signs of $Y_2$ and $Y_3$.
The first one-parameter family gives this relation matrix (writing $\lambda$ for $X_2$):
  \[
  \left[ \begin{array}{@{\;}cccccccc@{\;}} 
  1 & \pm i & 0 & 0 & 0 & 0 & 0 & 0 \\ 
  0 & 0 & 1 & 0 & \lambda & \mp \sqrt{1{-}\lambda^2} & 0 & 0 \\ 
  0 & 0 & 0 & 1 & \pm \sqrt{1{-}\lambda^2} & \lambda & 0 & 0 \\ 
  0 & 0 & 0 & 0 & 0 & 0 & 1 & \pm i
  \end{array} \right] 
  \]
We leave it to the reader to write down the relations between the operations $\opr, \opl$.
\end{example}


\section{Self-duality for two associative operations}
\label{associativesection}

\begin{definition}
The operad with two associative operations has the relation matrix $[A]$ 
whose row space is a 2-dimensional subspace $A \subset \calO(2)$:
  \[
  [A] =
  \left[
  \begin{array}{@{\,}r@{\;\;}r@{\;\;}r@{\;\;}r@{\;\;}r@{\;\;}r@{\;\;}r@{\;\;}r@{\,}}
  1 & 0 & 0 & 0 & -1 & 0 & 0 &  0 \\
  0 & 0 & 0 & 1 &  0 & 0 & 0 & -1
  \end{array}
  \right]
  \qquad
  \begin{array}{l}
  ( x_1 \opr x_2 ) \opr x_3 - x_1 \opr ( x_2 \opr x_3 ) \equiv 0
  \\
  ( x_1 \opl x_2 ) \opl x_3 - x_1 \opl ( x_2 \opl x_3 ) \equiv 0 
  \end{array} 
  \]
Linear combinations of $\opr, \opl$ are associative if and only if the operations satisfy 
the compatibility relation (Definition \ref{2opdefinition}).
\end{definition}

Let $[R]$ be a relation matrix of rank 4 for an operad with two binary operations.
$[R]$ is in RCF with leading 1s in columns $j_1, j_2, j_3, j_4$.
Since the operations are associative, $A \subset R$, and stacking $[R]$
on top of $[A]$ gives a matrix of rank 4.
Since $A$ has leading 1s in columns 1 and 4, we have $\{ 1, 4 \} \subset \{ j_1, j_2, j_3, j_4 \}$.
There are $\binom{6}{2} = 15$ cases for the other two columns;
see Table \ref{15relmats}.

\begin{table}[ht]
\[
\begin{array}{r@{\;\;}r@{\;\;}r} 
1\!\! 
\left[ 
\begin{array}{@{\,}c@{\;}c@{\;}c@{\;}c@{\;}c@{\;}c@{\;}c@{\;}c@{\,}} 
  1 &   0 &   0 &   0 & W_1 & X_1 & Y_1 & Z_1 \\[-0.5mm] 
  0 &   1 &   0 &   0 & W_2 & X_2 & Y_2 & Z_2 \\[-0.5mm] 
  0 &   0 &   1 &   0 & W_3 & X_3 & Y_3 & Z_3 \\[-0.5mm] 
  0 &   0 &   0 &   1 & W_4 & X_4 & Y_4 & Z_4 
\end{array} 
\right] 
& 
2\!\! 
\left[ 
\begin{array}{@{\,}c@{\;}c@{\;}c@{\;}c@{\;}c@{\;}c@{\;}c@{\;}c@{\,}} 
  1 &   0 & W_1 &   0 &   0 & X_1 & Y_1 & Z_1 \\[-0.5mm] 
  0 &   1 & W_2 &   0 &   0 & X_2 & Y_2 & Z_2 \\[-0.5mm] 
  0 &   0 &   0 &   1 &   0 & X_3 & Y_3 & Z_3 \\[-0.5mm] 
  0 &   0 &   0 &   0 &   1 & X_4 & Y_4 & Z_4 
\end{array} 
\right] 
& 
3\!\! 
\left[ 
\begin{array}{@{\,}c@{\;}c@{\;}c@{\;}c@{\;}c@{\;}c@{\;}c@{\;}c@{\,}} 
  1 &   0 & W_1 &   0 & X_1 &   0 & Y_1 & Z_1 \\[-0.5mm] 
  0 &   1 & W_2 &   0 & X_2 &   0 & Y_2 & Z_2 \\[-0.5mm] 
  0 &   0 &   0 &   1 & X_3 &   0 & Y_3 & Z_3 \\[-0.5mm] 
  0 &   0 &   0 &   0 &   0 &   1 & Y_4 & Z_4 
\end{array} 
\right] 
\\[8mm] 
4\!\! 
\left[ 
\begin{array}{@{\,}c@{\;}c@{\;}c@{\;}c@{\;}c@{\;}c@{\;}c@{\;}c@{\,}} 
  1 &   0 & W_1 &   0 & X_1 & Y_1 &   0 & Z_1 \\[-0.5mm] 
  0 &   1 & W_2 &   0 & X_2 & Y_2 &   0 & Z_2 \\[-0.5mm] 
  0 &   0 &   0 &   1 & X_3 & Y_3 &   0 & Z_3 \\[-0.5mm] 
  0 &   0 &   0 &   0 &   0 &   0 &   1 & Z_4 
\end{array} 
\right] 
& 
5\!\! 
\left[ 
\begin{array}{@{\,}c@{\;}c@{\;}c@{\;}c@{\;}c@{\;}c@{\;}c@{\;}c@{\,}} 
  1 &   0 & W_1 &   0 & X_1 & Y_1 & Z_1 &   0 \\[-0.5mm] 
  0 &   1 & W_2 &   0 & X_2 & Y_2 & Z_2 &   0 \\[-0.5mm] 
  0 &   0 &   0 &   1 & X_3 & Y_3 & Z_3 &   0 \\[-0.5mm] 
  0 &   0 &   0 &   0 &   0 &   0 &   0 &   1 
\end{array} 
\right] 
& 
6\!\! 
\left[ 
\begin{array}{@{\,}c@{\;}c@{\;}c@{\;}c@{\;}c@{\;}c@{\;}c@{\;}c@{\,}} 
  1 & W_1 &   0 &   0 &   0 & X_1 & Y_1 & Z_1 \\[-0.5mm] 
  0 &   0 &   1 &   0 &   0 & X_2 & Y_2 & Z_2 \\[-0.5mm] 
  0 &   0 &   0 &   1 &   0 & X_3 & Y_3 & Z_3 \\[-0.5mm] 
  0 &   0 &   0 &   0 &   1 & X_4 & Y_4 & Z_4 
\end{array} 
\right] 
\\[8mm] 
7\!\! 
\left[ 
\begin{array}{@{\,}c@{\;}c@{\;}c@{\;}c@{\;}c@{\;}c@{\;}c@{\;}c@{\,}} 
  1 & W_1 &   0 &   0 & X_1 &   0 & Y_1 & Z_1 \\[-0.5mm] 
  0 &   0 &   1 &   0 & X_2 &   0 & Y_2 & Z_2 \\[-0.5mm] 
  0 &   0 &   0 &   1 & X_3 &   0 & Y_3 & Z_3 \\[-0.5mm] 
  0 &   0 &   0 &   0 &   0 &   1 & Y_4 & Z_4 
\end{array} 
\right] 
& 
8\!\! 
\left[ 
\begin{array}{@{\,}c@{\;}c@{\;}c@{\;}c@{\;}c@{\;}c@{\;}c@{\;}c@{\,}} 
  1 & W_1 &   0 &   0 & X_1 & Y_1 &   0 & Z_1 \\[-0.5mm] 
  0 &   0 &   1 &   0 & X_2 & Y_2 &   0 & Z_2 \\[-0.5mm] 
  0 &   0 &   0 &   1 & X_3 & Y_3 &   0 & Z_3 \\[-0.5mm] 
  0 &   0 &   0 &   0 &   0 &   0 &   1 & Z_4 
\end{array} 
\right] 
& 
9\!\! 
\left[ 
\begin{array}{@{\,}c@{\;}c@{\;}c@{\;}c@{\;}c@{\;}c@{\;}c@{\;}c@{\,}} 
  1 & W_1 &   0 &   0 & X_1 & Y_1 & Z_1 &   0 \\[-0.5mm] 
  0 &   0 &   1 &   0 & X_2 & Y_2 & Z_2 &   0 \\[-0.5mm] 
  0 &   0 &   0 &   1 & X_3 & Y_3 & Z_3 &   0 \\[-0.5mm] 
  0 &   0 &   0 &   0 &   0 &   0 &   0 &   1 
\end{array} 
\right] 
\\[8mm] 
10\!\! 
\left[ 
\begin{array}{@{\,}c@{\;}c@{\;}c@{\;}c@{\;}c@{\;}c@{\;}c@{\;}c@{\,}} 
  1 & W_1 & X_1 &   0 &   0 &   0 & Y_1 & Z_1 \\[-0.5mm] 
  0 &   0 &   0 &   1 &   0 &   0 & Y_2 & Z_2 \\[-0.5mm] 
  0 &   0 &   0 &   0 &   1 &   0 & Y_3 & Z_3 \\[-0.5mm] 
  0 &   0 &   0 &   0 &   0 &   1 & Y_4 & Z_4 
\end{array} 
\right] 
& 
11\!\! 
\left[ 
\begin{array}{@{\,}c@{\;}c@{\;}c@{\;}c@{\;}c@{\;}c@{\;}c@{\;}c@{\,}} 
  1 & W_1 & X_1 &   0 &   0 & Y_1 &   0 & Z_1 \\[-0.5mm] 
  0 &   0 &   0 &   1 &   0 & Y_2 &   0 & Z_2 \\[-0.5mm] 
  0 &   0 &   0 &   0 &   1 & Y_3 &   0 & Z_3 \\[-0.5mm] 
  0 &   0 &   0 &   0 &   0 &   0 &   1 & Z_4 
\end{array} 
\right] 
& 
12\!\! 
\left[ 
\begin{array}{@{\,}c@{\;}c@{\;}c@{\;}c@{\;}c@{\;}c@{\;}c@{\;}c@{\,}} 
  1 & W_1 & X_1 &   0 &   0 & Y_1 & Z_1 &   0 \\[-0.5mm] 
  0 &   0 &   0 &   1 &   0 & Y_2 & Z_2 &   0 \\[-0.5mm] 
  0 &   0 &   0 &   0 &   1 & Y_3 & Z_3 &   0 \\[-0.5mm] 
  0 &   0 &   0 &   0 &   0 &   0 &   0 &   1 
\end{array} 
\right] 
\\[8mm] 
13\!\! 
\left[ 
\begin{array}{@{\,}c@{\;}c@{\;}c@{\;}c@{\;}c@{\;}c@{\;}c@{\;}c@{\,}} 
  1 & W_1 & X_1 &   0 & Y_1 &   0 &   0 & Z_1 \\[-0.5mm] 
  0 &   0 &   0 &   1 & Y_2 &   0 &   0 & Z_2 \\[-0.5mm] 
  0 &   0 &   0 &   0 &   0 &   1 &   0 & Z_3 \\[-0.5mm] 
  0 &   0 &   0 &   0 &   0 &   0 &   1 & Z_4 
\end{array} 
\right] 
& 
14\!\! 
\left[ 
\begin{array}{@{\,}c@{\;}c@{\;}c@{\;}c@{\;}c@{\;}c@{\;}c@{\;}c@{\,}} 
  1 & W_1 & X_1 &   0 & Y_1 &   0 & Z_1 &   0 \\[-0.5mm] 
  0 &   0 &   0 &   1 & Y_2 &   0 & Z_2 &   0 \\[-0.5mm] 
  0 &   0 &   0 &   0 &   0 &   1 & Z_3 &   0 \\[-0.5mm] 
  0 &   0 &   0 &   0 &   0 &   0 &   0 &   1 
\end{array} 
\right] 
& 
15\!\! 
\left[ 
\begin{array}{@{\,}c@{\;}c@{\;}c@{\;}c@{\;}c@{\;}c@{\;}c@{\;}c@{\,}} 
  1 & W_1 & X_1 &   0 & Y_1 & Z_1 &   0 &   0 \\[-0.5mm] 
  0 &   0 &   0 &   1 & Y_2 & Z_2 &   0 &   0 \\[-0.5mm] 
  0 &   0 &   0 &   0 &   0 &   0 &   1 &   0 \\[-0.5mm] 
  0 &   0 &   0 &   0 &   0 &   0 &   0 &   1 
\end{array} 
\right] 
\end{array} 
\]
\caption{Relation matrices for two associative operations}
\label{15relmats}
\end{table}

\begin{definition}
To obtain necessary and sufficient conditions for the operations to be associative, 
we stack $[R]$ on top of $[A]$, and use the leading 1s in $[R]$ to eliminate the nonzero entries
in columns $j_1,\dots,j_4$ in $[A]$.
Then rows 5, 6 are zero if and only if $R + A = R$, that is $A \subset R$.
The equations obtained by setting the entries in rows 5, 6 to zero are the
\textbf{associativity conditions} on the parameters.
\end{definition}

\begin{example}
Consider the relation matrix $[R]$ in case 1:
  \[
  \left[
  \begin{array}{c}
  R
  \\
  A
  \end{array}
  \right]
  =
  \left[ 
  \begin{array}{@{\,}c@{\;}c@{\;}c@{\;}c@{\;}c@{\;}c@{\;}c@{\;}c@{\,}} 
  1 &   0 &   0 &   0 & W_1 & X_1 & Y_1 & Z_1 \\ 
  0 &   1 &   0 &   0 & W_2 & X_2 & Y_2 & Z_2 \\ 
  0 &   0 &   1 &   0 & W_3 & X_3 & Y_3 & Z_3 \\ 
  0 &   0 &   0 &   1 & W_4 & X_4 & Y_4 & Z_4 \\
  \midrule  
  1  &  0  &  0  &  0  & -1       &  0       &  0       &  0       \\
  0  &  0  &  0  &  1  &  0       &  0       &  0       & -1
  \end{array}
  \right]
  \longrightarrow
  \left[
  \begin{array}{@{\,}c@{\;}c@{\;}c@{\;}c@{\;}c@{\;\;}c@{\;\;}c@{\;\;}c@{\,}} 
  1 &   0 &   0 &   0 & W_1 & X_1 & Y_1 & Z_1 \\ 
  0 &   1 &   0 &   0 & W_2 & X_2 & Y_2 & Z_2 \\ 
  0 &   0 &   1 &   0 & W_3 & X_3 & Y_3 & Z_3 \\ 
  0 &   0 &   0 &   1 & W_4 & X_4 & Y_4 & Z_4 \\
  \midrule
  0  &  0  &  0  &  0  &  -W_1{-}1  &  -X_1   &  -Y_1   &  -Z_1    \\
  0  &  0  &  0  &  0  &  -W_4   &  -X_4  &  -Y_4  &  -Z_4{-}1
  \end{array}
  \right]
  \]
Hence $A \subset R$ if and only if 
$W_1, Z_4 = -1$ and $X_1, Y_1, Z_1, W_4, X_4, Y_4 = 0$,
which means that rows 1 and 4 of $[R]$ coincide with the two rows of $[A]$.
\end{example}

The relation matrices obtained after applying the associativity conditions appear 
in Table \ref{15newrelmats}.
The number of parameters in each case has dropped.

\begin{table}[ht]
\[
\begin{array}{r@{}l@{\;}r@{}l@{\;}r@{}l} 
1\! 
&\left[ 
\begin{array}{@{}c@{\;}c@{\;}c@{\;}c@{\;}c@{\;}c@{\;}c@{\;}c@{}} 
  1 &   0 &   0 &   0 &  -1 &   0 &   0 &   0 \\[-0.5mm] 
  0 &   1 &   0 &   0 & W_2 & X_2 & Y_2 & Z_2 \\[-0.5mm] 
  0 &   0 &   1 &   0 & W_3 & X_3 & Y_3 & Z_3 \\[-0.5mm] 
  0 &   0 &   0 &   1 &   0 &   0 &   0 &  -1 
\end{array} 
\right] 
& 
2\! 
&\left[ 
\begin{array}{@{}c@{\;}c@{\;}c@{\;}c@{\;}c@{\;}c@{\;}c@{\;}c@{}} 
  1 &   0 &   0 &   0 &   0 & X_1 & Y_1 & Z_1 \\[-0.5mm] 
  0 &   1 & W_2 &   0 &   0 & X_2 & Y_2 & Z_2 \\[-0.5mm] 
  0 &   0 &   0 &   1 &   0 &   0 &   0 &  -1 \\[-0.5mm] 
  0 &   0 &   0 &   0 &   1 & X_1 & Y_1 & Z_1 
\end{array} 
\right] 
& 
3\! 
&\left[ 
\begin{array}{@{}c@{\;}c@{\;}c@{\;}c@{\;}c@{\;}c@{\;}c@{\;}c@{}} 
  1 &   0 &   0 &   0 &  -1 &   0 &   0 &   0 \\[-0.5mm] 
  0 &   1 & W_2 &   0 & X_2 &   0 & Y_2 & Z_2 \\[-0.5mm] 
  0 &   0 &   0 &   1 &   0 &   0 &   0 &  -1 \\[-0.5mm] 
  0 &   0 &   0 &   0 &   0 &   1 & Y_4 & Z_4 
\end{array} 
\right] 
\\[8mm] 
4\! 
&\left[ 
\begin{array}{@{}c@{\;}c@{\;}c@{\;}c@{\;}c@{\;}c@{\;}c@{\;}c@{}} 
  1 &   0 &   0 &   0 &  -1 &   0 &   0 &   0 \\[-0.5mm] 
  0 &   1 & W_2 &   0 & X_2 & Y_2 &   0 & Z_2 \\[-0.5mm] 
  0 &   0 &   0 &   1 &   0 &   0 &   0 &  -1 \\[-0.5mm] 
  0 &   0 &   0 &   0 &   0 &   0 &   1 & Z_4 
\end{array} 
\right] 
& 
5\! 
&\left[ 
\begin{array}{@{}c@{\;}c@{\;}c@{\;}c@{\;}c@{\;}c@{\;}c@{\;}c@{}} 
  1 &   0 &   0 &   0 &  -1 &   0 &   0 &   0 \\[-0.5mm] 
  0 &   1 & W_2 &   0 & X_2 & Y_2 & Z_2 &   0 \\[-0.5mm] 
  0 &   0 &   0 &   1 &   0 &   0 &   0 &   0 \\[-0.5mm] 
  0 &   0 &   0 &   0 &   0 &   0 &   0 &   1 
\end{array} 
\right] 
& 
6\! 
&\left[ 
\begin{array}{@{}c@{\;}c@{\;}c@{\;}c@{\;}c@{\;}c@{\;}c@{\;}c@{}} 
  1 &   0 &   0 &   0 &   0 & X_1 & Y_1 & Z_1 \\[-0.5mm] 
  0 &   0 &   1 &   0 &   0 & X_2 & Y_2 & Z_2 \\[-0.5mm] 
  0 &   0 &   0 &   1 &   0 &   0 &   0 &  -1 \\[-0.5mm] 
  0 &   0 &   0 &   0 &   1 & X_1 & Y_1 & Z_1 
\end{array} 
\right] 
\\[8mm] 
7\! 
&\left[ 
\begin{array}{@{}c@{\;}c@{\;}c@{\;}c@{\;}c@{\;}c@{\;}c@{\;}c@{}} 
  1 &   0 &   0 &   0 &  -1 &   0 &   0 &   0 \\[-0.5mm] 
  0 &   0 &   1 &   0 & X_2 &   0 & Y_2 & Z_2 \\[-0.5mm] 
  0 &   0 &   0 &   1 &   0 &   0 &   0 &  -1 \\[-0.5mm] 
  0 &   0 &   0 &   0 &   0 &   1 & Y_4 & Z_4 
\end{array} 
\right] 
& 
8\! 
&\left[ 
\begin{array}{@{}c@{\;}c@{\;}c@{\;}c@{\;}c@{\;}c@{\;}c@{\;}c@{}} 
  1 &   0 &   0 &   0 &  -1 &   0 &   0 &   0 \\[-0.5mm] 
  0 &   0 &   1 &   0 & X_2 & Y_2 &   0 & Z_2 \\[-0.5mm] 
  0 &   0 &   0 &   1 &   0 &   0 &   0 &  -1 \\[-0.5mm] 
  0 &   0 &   0 &   0 &   0 &   0 &   1 & Z_4 
\end{array} 
\right] 
& 
9\! 
&\left[ 
\begin{array}{@{}c@{\;}c@{\;}c@{\;}c@{\;}c@{\;}c@{\;}c@{\;}c@{}} 
  1 &   0 &   0 &   0 &  -1 &   0 &   0 &   0 \\[-0.5mm] 
  0 &   0 &   1 &   0 & X_2 & Y_2 & Z_2 &   0 \\[-0.5mm] 
  0 &   0 &   0 &   1 &   0 &   0 &   0 &   0 \\[-0.5mm] 
  0 &   0 &   0 &   0 &   0 &   0 &   0 &   1 
\end{array} 
\right] 
\\[8mm] 
10\! 
&\left[ 
\begin{array}{@{}c@{\;}c@{\;}c@{\;}c@{\;}c@{\;}c@{\;}c@{\;}c@{}} 
  1 &   0 &   0 &   0 &   0 &   0 & Y_1 & Z_1 \\[-0.5mm] 
  0 &   0 &   0 &   1 &   0 &   0 &   0 &  -1 \\[-0.5mm] 
  0 &   0 &   0 &   0 &   1 &   0 & Y_1 & Z_1 \\[-0.5mm] 
  0 &   0 &   0 &   0 &   0 &   1 & Y_4 & Z_4 
\end{array} 
\right] 
& 
11\! 
&\left[ 
\begin{array}{@{}c@{\;}c@{\;}c@{\;}c@{\;}c@{\;}c@{\;}c@{\;}c@{}} 
  1 &   0 &   0 &   0 &   0 & Y_1 &   0 & Z_1 \\[-0.5mm] 
  0 &   0 &   0 &   1 &   0 &   0 &   0 &  -1 \\[-0.5mm] 
  0 &   0 &   0 &   0 &   1 & Y_1 &   0 & Z_1 \\[-0.5mm] 
  0 &   0 &   0 &   0 &   0 &   0 &   1 & Z_4 
\end{array} 
\right] 
& 
12\! 
&\left[ 
\begin{array}{@{}c@{\;}c@{\;}c@{\;}c@{\;}c@{\;}c@{\;}c@{\;}c@{}} 
  1 &   0 &   0 &   0 &   0 & Y_1 & Z_1 &   0 \\[-0.5mm] 
  0 &   0 &   0 &   1 &   0 &   0 &   0 &   0 \\[-0.5mm] 
  0 &   0 &   0 &   0 &   1 & Y_1 & Z_1 &   0 \\[-0.5mm] 
  0 &   0 &   0 &   0 &   0 &   0 &   0 &   1 
\end{array} 
\right] 
\\[8mm] 
13\! 
&\left[ 
\begin{array}{@{}c@{\;}c@{\;}c@{\;}c@{\;}c@{\;}c@{\;}c@{\;}c@{}} 
  1 &   0 &   0 &   0 &  -1 &   0 &   0 &   0 \\[-0.5mm] 
  0 &   0 &   0 &   1 &   0 &   0 &   0 &  -1 \\[-0.5mm] 
  0 &   0 &   0 &   0 &   0 &   1 &   0 & Z_3 \\[-0.5mm] 
  0 &   0 &   0 &   0 &   0 &   0 &   1 & Z_4 
\end{array} 
\right] 
& 
14\! 
&\left[ 
\begin{array}{@{}c@{\;}c@{\;}c@{\;}c@{\;}c@{\;}c@{\;}c@{\;}c@{}} 
  1 &   0 &   0 &   0 &  -1 &   0 &   0 &   0 \\[-0.5mm] 
  0 &   0 &   0 &   1 &   0 &   0 &   0 &   0 \\[-0.5mm] 
  0 &   0 &   0 &   0 &   0 &   1 & Z_3 &   0 \\[-0.5mm] 
  0 &   0 &   0 &   0 &   0 &   0 &   0 &   1 
\end{array} 
\right] 
& 
15\! 
&\left[ 
\begin{array}{@{}c@{\;}c@{\;}c@{\;}c@{\;}c@{\;}c@{\;}c@{\;}c@{}} 
  1 &   0 &   0 &   0 &  -1 &   0 &   0 &   0 \\[-0.5mm] 
  0 &   0 &   0 &   1 &   0 &   0 &   0 &   0 \\[-0.5mm] 
  0 &   0 &   0 &   0 &   0 &   0 &   1 &   0 \\[-0.5mm] 
  0 &   0 &   0 &   0 &   0 &   0 &   0 &   1 
\end{array} 
\right] 
\end{array}
\]
\caption{Relation matrices after applying associativity conditions}
\label{15newrelmats}
\end{table}

It remains to use Loday's algorithm (see Table \ref{lodayalg}) to determine which values of the parameters 
produce self-dual operads.

\begin{lemma} \label{lemma:k=6}
For cases $6,\dots,15$ no values of the parameters imply self-duality.
\end{lemma}

\begin{proof}
In these cases $T$ contains a nonzero scalar; the rest follows Lemma \ref{lemma35notself-dual}.
\end{proof}

\begin{theorem} \label{maintheoremforthissection}
The quadratic nonsymmetric operad $\calQ$ with two associative binary operationsis self-dual
if and only if its relation matrix $[R]$ is one of:
  \[
  \left[ 
  \begin{array}{@{\,}c@{\;}c@{\;}c@{\;}c@{\;}c@{\;}c@{\;}c@{\;}c@{\,}}
  1 &   0 &   0 &   0 &  -1 &   0 &   0 &   0 \\[-0.5mm] 
  0 &   1 &   0 &   0 &   0 & \pm \lambda & \pm \sqrt{1{-}\lambda^2} &   0 \\[-0.5mm] 
  0 &   0 &   1 &   0 &   0 & \pm \sqrt{1{-}\lambda^2} & \lambda &   0 \\[-0.5mm] 
  0 &   0 &   0 &   1 &   0 &   0 &   0 &  -1 
  \end{array}
  \right]
  \;\;
  (\lambda \in \mathbb{F})
  \qquad\qquad
  \left[
  \begin{array}{@{\,}c@{\;}c@{\;}c@{\;}c@{\;}c@{\;}c@{\;}c@{\;}c@{\,}}
  1  &  0  &  0  &  0  &  -1  &  0  &  0  &  0 \\[-0.5mm]
  0  &  1  &  \pm i  &  0  &  0  &  0  &  0  &  0 \\[-0.5mm]
  0  &  0  &  0  &  1  &  0  &  0  &  0  &  -1 \\[-0.5mm]
  0  &  0  &  0  &  0  &  0  &  1  &  \pm i  &  0
  \end{array}
  \right] 
  \]
\end{theorem}

\begin{proof}
By Lemma \ref{lemma:k=6}, self-dual associative operads exist only in cases $1,\dots,5$.

\noindent $\bullet$
\emph{Case 1:}
We obtain (omitting zero columns from $-T$):
  \begin{align*}
  \left[ \begin{array}{@{}c@{}} R \\ R^\pm \end{array} \right]
  &=
  \left[ 
  \begin{array}{cccccccc}
  1 &   0 &   0 &   0 &  -1 &   0 &   0 &   0 \\[-0.5mm] 
  0 &   1 &   0 &   0 & W_2 & X_2 & Y_2 & Z_2 \\[-0.5mm] 
  0 &   0 &   1 &   0 & W_3 & X_3 & Y_3 & Z_3 \\[-0.5mm] 
  0 &   0 &   0 &   1 &   0 &   0 &   0 &  -1 \\[-0.5mm]
  \midrule
  -1 & W_2 & W_3 & 0 & 1 & 0 & 0 & 0 \\[-0.5mm]
  0 & X_2 & X_3 & 0 & 0 & 1 & 0 & 0 \\[-0.5mm]
  0 & Y_2 & Y_3 & 0 & 0 & 0 & 1 & 0 \\[-0.5mm]
  0 & Z_2 & Z_3 & -1 & 0 & 0 & 0 & 1
  \end{array}
  \right]
  \longrightarrow
  \left[ \begin{array}{@{}r@{}} R \\ T \end{array} \right], 
  \\
  -T 
  &=
  \left[ 
  \begin{array}{@{\,}c@{\quad}c@{\quad}c@{\quad}c@{\,}}
  W_2^2 {+} W_3^2 & W_2 X_2 {+} W_3 X_3 & W_2 Y_2 {+} W_3 Y_3 & W_2 Z_2 {+} W_3 Z_3 
  \\[2pt]
  W_2 X_2 {+} W_3 X_3 & X_2^2 {+} X_3^2 {-} 1 & X_2 Y_2 {+} X_3 Y_3 & X_2 Z_2 {+} X_3 Z_3 
  \\[2pt]
  W_2 Y_2 {+} W_3 Y_3 & X_2 Y_2 {+} X_3 Y_3 & Y_2^2 {+} Y_3^2 {-} 1 & Y_2 Z_2 {+} Y_3 Z_3 
  \\[2pt]
  W_2 Z_2 {+} W_3 Z_3 & X_2 Z_2 {+} X_3 Z_3 & Y_2 Z_2 {+} Y_3 Z_3 & Z_2^2 {+} Z_3^2
  \end{array}
  \right]
  \end{align*}
The entries of $[T]$ generate an ideal $\mathcal{I}$ in the polynomial ring $\Phi[\mathcal{P}]$
where $\mathcal{P}$ is the set of 8 parameters.
The zero set $V(\mathcal{I})$ consists of the parameter values for which 
the relation matrix $[R]$ defines a self-dual operad.
The Gr\"obner basis for $\mathcal{I}$ is
\begin{align*}
&
Z_3, \quad Z_2, \quad W_3, \quad W_2, \quad Y_2^2+Y_3^2-1, \quad X_2Y_2+X_3Y_3, \quad X_3^2+Y_3^2-1,
\\
&
X_2X_3+Y_2Y_3, \quad X_2^2-Y_3^2, \quad X_3Y_2Y_3-X_2Y_3^2+X_2.
\end{align*}
We obtain a one-parameter set of solutions; the signs may be chosen independently:
\begin{align*}
&
W_2 = 0, \quad
W_3 = 0, \quad
X_2 = \pm Y_3, \quad
X_3 = \pm \sqrt{1-Y_3^2},
\\[-1mm]
&
Y_2 = \pm \sqrt{1-Y_3^2}, \quad
Y_3 = \text{free}, \quad
Z_2 = 0, \quad
Z_3 = 0.
\end{align*}
With these values the relation matrix $[R_1]$ takes the indicated form.

\noindent $\bullet$
\emph{Case 2:}
In this case, after deleting the columns which are zero, we obtain
  \[
  [T]
  =
  \left[ 
  \begin{array}{@{\;}c@{\quad}c@{\quad}c@{\quad}c@{\;}}
  W_2^2+1 & W_2 X_2 & W_2 Y_2 & W_2 Z_2 \\
  -W_2 X_2 & -X_2^2+1 & -X_2 Y_2 & -X_2 Z_2 \\
  -W_2 Y_2 & -X_2 Y_2 & -Y_2^2+1 & -Y_2 Z_2 \\
  -W_2 Z_2 & -X_2 Z_2 & -Y_2 Z_2 & -Z_2^2
  \end{array}
  \right]
  \]
The ideal generated by the entries of $[T]$ has Gr\"obner basis $\{1\}$: no solutions.

\noindent $\bullet$
\emph{Case 3:}
After deleting the zero columns, we obtain
  \[
  [T]
  =
  \left[ 
  \begin{array}{@{\;}c@{\quad}c@{\quad}c@{\quad}c@{\;}}
  W_2^2+1 & W_2 X_2 & W_2 Y_2 & W_2 Z_2 \\
  -W_2 X_2 & -X_2^2 & -X_2 Y_2 & -X_2 Z_2 \\
  -W_2 Y_2 & -X_2 Y_2 & -Y_2^2+Y_4^2+1 & -Y_2 Z_2+Y_4 Z_4 \\
  -W_2 Z_2 & -X_2 Z_2 & -Y_2 Z_2+Y_4 Z_4 & -Z_2^2+Z_4^2
  \end{array}
  \right]
  \]
The Gr\"obner basis for the ideal generated by the entries is
\[
Z_4, \quad Z_2, \quad Y_2, \quad X_2, \quad Y_4^2+1, \quad W_2^2+1.
\]
With the zero set of this ideal, the relation matrix $[R_3]$ takes the indicated form.

\noindent $\bullet$
\emph{Cases 4, 5:}
The ideal generated by the entries of $[T]$ has Gr\"obner basis $\{1\}$.
\end{proof}
  

\section*{Acknowledgements}

Murray Bremner was supported by a Discovery Grant from NSERC, the Natural Sciences and Engineering
Research Council of Canada.
Juana S\'anchez-Ortega was supported by a CSUR grant from the NRF, the National Research Foundation
of South Africa.
The authors thank Vladimir Dotsenko for pointing out the connection between relation matrices
and Schubert calculus on Grassmannians.



\begin{thebibliography}{99}

\bibitem{BD}
\textsc{M. R. Bremner, V. Dotsenko}:
\emph{Algebraic Operads: An Algorithmic Companion}.
CRC Press, Boca Raton, 2016.

\bibitem{Fulton}
\textsc{W. Fulton}:
\emph{Young Tableaux, with Applications to Representation Theory and Geometry}. 
London Mathematical Society Student Texts, 35. 
Cambridge University Press, 1997.

\bibitem{Hiller}
\textsc{H. Hiller}:
\emph{Geometry of Coxeter Groups}. 
Research Notes in Mathematics, 54. 
Pitman (Advanced Publishing Program), Boston \& London, 1982. 

\bibitem{Loday1995}
\textsc{J.-L. Loday}: 
Alg\`ebres ayant deux op\'erations associatives (dig\`ebres).
C. R. Acad. Sci. Paris S\'er. I Math. 
321 (1995), no. 2, 141--146. 

\bibitem{Loday2001}
\textsc{J.-L. Loday}: 
Dialgebras. 
\emph{Dialgebras and Related Operads}, pages 7--66.
Lecture Notes in Mathematics, 1763.
Springer, Berlin, 2001. 

\bibitem{LV}
\textsc{J.-L. Loday, B. Vallette}:
\emph{Algebraic Operads}.
Grundlehren der mathematischen Wissenschaften, 346.
Springer, Heidelberg, 2012.

\bibitem{MR}
\textsc{M. Markl, E. Remm}:
(Non-)Koszulness of operads for $n$-ary algebras, galgalim and other curiosities.
\emph{J. Homotopy Relat. Struct.}
10 (2015), no.~4, 939--969.

\bibitem{Osborn}
\textsc{J. M. Osborn}:
Identities of non-associative algebras. 
\emph{Canad. J. Math.} 17 (1965) 78--92. 

\bibitem{Zinbiel}
\textsc{G. W. Zinbiel (J.-L. Loday)}:
Encyclopedia of types of algebras 2010.
\emph{Operads and Universal Algebra}, pages 217--297.
Nankai Ser. Pure Appl. Math. Theoret. Phys., 9.
World Sci. Publ., Hackensack, NJ, 2012. 

\end{thebibliography}
\end{document}